\documentclass{amsart}
\usepackage{amsmath, amsfonts,amssymb,amsthm,mathrsfs,mathtools}
\usepackage{tikz}
\usetikzlibrary{matrix,arrows}
\usepackage{comment}

\theoremstyle{plain}
\newtheorem{theorem}{Theorem}[section]

\newtheorem{proposition}[theorem]{Proposition}
\newtheorem{lemma}[theorem]{Lemma}

\newtheorem{corollary}[theorem]{Corollary}

\theoremstyle{definition}
\newtheorem{definition}[theorem]{Definition}
\newtheorem{remark}[theorem]{Remark}

\DeclareMathOperator{\Tr}{Tr}

\DeclareMathOperator{\HP}{HP}
\DeclareMathOperator{\NP}{NP}

\DeclareMathOperator{\Matrix}{Mat}

\DeclareMathOperator{\Sym}{Sym}
\DeclareMathOperator{\sgn}{sgn}
\DeclareMathOperator{\lcm}{lcm}

\DeclareMathOperator{\sNP}{sNP}
\DeclareMathOperator{\gNP}{gNP}
\DeclareMathOperator{\GNP}{GNP}
\DeclareMathOperator{\ord}{ord}
\DeclareMathOperator{\rig}{rig}
\DeclareMathOperator{\wt}{wt}

\renewcommand{\hat}{\widehat}

\renewcommand{\hat}{\widehat}

\renewcommand{\bar}{\overline}

\newcommand{\pfloor}[1]{\lfloor #1 \rfloor}

\renewcommand{\bar}{\overline}

\newcommand{\A}{\mathbf{A}}

\newcommand{\C}{\mathbf{C}}
\newcommand{\F}{\mathbf{F}}
\renewcommand{\P}{\mathbf{P}}
\newcommand{\Q}{\mathbf{Q}}
\newcommand{\Z}{\mathbf{Z}}
\newcommand{\R}{\mathbf{R}}
\newcommand{\Gm}{\mathbf{G_m}}
\newcommand{\Gmrig}{\mathbf{G_{m,\rig}}}

\newcommand{\Zp}{\mathbf{Z}_p}
\newcommand{\Cp}{\mathbf{C}_p}

\newcommand{\bv}{{\vec{v}}}
\newcommand{\bx}{{\vec{x}}}
\newcommand{\by}{{\vec{y}}}

\newcommand{\br}{{\vec{r}}}
\newcommand{\bi}{{\vec{i}}}
\newcommand{\bj}{{\vec{j}}}

\newcommand{\ha}{\hat{a}}
\newcommand{\bw}{{\vec{w}}}

\newcommand{\cU}{\mathcal{U}}
\newcommand{\cD}{\mathcal{D}}
\newcommand{\cS}{\mathcal{S}}
\newcommand{\cK}{\mathcal{K}}
\newcommand{\cF}{\mathcal{F}}

\newcommand{\cR}{\mathcal{R}}
\newcommand{\cW}{\mathcal{W}}
\newcommand{\cA}{\mathcal{A}}

\begin{document}

\title[Slopes for Artin-Schreier-Witt tower in two variables]
{Generic Newton slopes for Artin-Schreier-Witt tower in two variables}

\author{Hui June Zhu}
\address{
Department of mathematics,
State University of New York at Buffalo,
Buffalo, NY 14260, USA}
\email{hjzhu@math.buffalo.edu}
\date{December 21 2016}

\maketitle

\begin{abstract}
If $\chi_m$ denotes any character of conductor $p^m$,
we prove the existence of global generic Newton slopes of $L$-functions
associated to 
$\chi_m$ of the Galois group $\Zp$ of a family of Artin-Schreier-Witt tower 
in a $2$-variable family.
Let $\A_\Delta$ denote the coefficient space of all polynomials 
$f(x_1,x_2)=\sum_{\vec{v}\in \Delta}a_{\vec{v}}x_1^{v_1}x_2^{v_2}$ 
in two variables
with $x_1$-degree $d_1$ and $x_2$-degree $d_2$.
We prove that there is a Zariski dense open subset $\cU_\Delta$ in $\A_\Delta$ 
defined over $\Q$ such that for any $f$ in $\cU_\Delta(\bar\Q)$
and for every prime $p$ large enough, 
the Newton slopes of $C$-function $C_f^*(\chi_m,s)$ of 
$f\bmod p$ are independent of the choice of $\chi_m$;
the Newton slopes of $L$-function $L_f^*(\chi_m,s)$ are
in weighted arithmetic progression of that of $L_f^*(\chi_1,s)$ for any $m>0$,
generalizing 1-variable result of Davis-Wan-Xiao. 
In this paper we also prove that for any
$f$ in $\cU_\Delta(\bar\Q)$, its corresponding eigencurve 
$\cA_f$ at $p$ is an infinite 
disjoint union of covers of a unit point and 
of degree $(2i-1)d_1d_2+1+\delta_\Delta$ 
for some $0\leq \delta_\Delta< d_1d_2$ over
the $p$-adic weight space for all $i\geq 1$. 
\end{abstract}

\tableofcontents

\section{Introduction}
\label{S:introduction}

Let $p$ be a prime and $q=p^a$ for some positive integer $a$.
Write $\bar{f}=\sum_{\bv\in\Delta}\bar{a}_\bv \bx^{\bv}$ where $\bar{a}_\bv\in \F_q$,
$\bx^\bv:=x_1^{v_1}\cdots x_n^{v_n}$, and 
$\Delta=\Delta(\bar{f})$ is the Newton polytope of $\bar{f}$ in $\R^n$.
Then the {\em Artin-Schreier-Witt tower} associated to $\bar{f}$ is the sequence of curves 
$\cdots\rightarrow A_m(f)\rightarrow A_{m-1}(f)\rightarrow \cdots \rightarrow A_2(f)\rightarrow A_1(f)$ 
over the projective line $\P^1/\F_q$ defined by the  following equations:
$$
A_m(f): \underline{y}^F-\underline{y} = \sum_{\bv\in\Delta} 
(\bar{a}_\bv \bx^\bv,0,\ldots,0) 
$$
where $\underline{y}_m=(y_0, y_1, \cdots,y_{m-1})$ 
is a Witt vector of length $m$, and 
$\bigstar^{F}$ is the endomorphism that raises each Witt coordinate to $p$-th power. 
Following standard notation we denote  $\wp(\underline{x}):=\underline{x}^F-\underline{x}$
for all Witt vector $\underline{x}$ of any length-$m$.
For simplicity, we assume $\sum_{\bv\in \Delta}a_\bv \bx^{\bv}\not\in \wp(\F_q(x_1,\ldots,x_n))$
and the Galois group of the above Artin-Schreier-Witt tower is $\Zp$.
For more information and proofs on Artin-Schreier-Witt extensions, 
see Lang's Algebra book \cite[Exercises of Chapter VI]{Lan02} 
or thorough details in Thomas' thesis \cite[Chapters 1 and 2]{Tho05}.

Let $\chi_m:\Zp\rightarrow \Cp^*$ be a character with {\em conductor} $p^{m}$ for some 
$m\in\Z_{\geq 1}$, namely, $\chi_m(1)=\zeta_{p^{m}}$ for some primitive $p^{m}$-th root 
of unity $\zeta_{p^m}$.
Let $\hat{f}$ be the Teichm\"uller lifting of $\bar{f}$ to $\Z_q$ coefficientwise.
The $L$-function of $\bar{f}$ with respect to $\chi_m$ is 
$$
L^*_f(\chi_m,s)=\prod_{\bx\in \Gm^n(\F_q)}\frac{1}{1-\chi_m(\Tr_{\Q_{q^{\deg(\bx)}}/\Q_p}(\hat{f}(\hat{\bx})) )s^{\deg(\bx)}}.
$$
From now on we assume $\bar{f}$ is regular 
with respect to $\Delta$
in the sense of 
Adolphson-Sperber \cite{AS89} 
so 
that 
$L^*_f(\chi_m,s)^{-1}$ is a polynomial of degree $p^{n(m-1)}n!V(\Delta)$
in variable $s$,
where $V(\Delta)$ is the volume of $\Delta$ in $\R^n$.
This is due to \cite{AS89} when $m=1$ and \cite[Theorem 1.3]{LW07}) when $m>1$. 

Main result in this paper is summarized in the 
following theorems. 
We write $\NP_q(-)$ for the 
$q$-adic Newton polygon of a polynomial in variable $s$ below
if $\bar{f}$ is over $\F_q$.
Let $\Delta$ be the Newton polytope with vertices $(0,0)$, $(d_1,0)$, $(0,d_2)$ and $(d_1,d_2)$,
where $d_1,d_2\geq 3$.
Let $\A_\Delta$ denote the space of all polynomials $f=\sum_{\bv\in \Delta} a_{\bv}\bx^{\bv}$.
It is known that for each prime $p$ there exists a generic Newton polygon 
for all polynomials $f$ in $\A_\Delta(\bar\F_p)$ for the traditional $L$-function $L_f^*(\chi_1,s)$. 
But these generic spaces over $\bar\F_p$ for different prime $p$ 
have no known connection. In this paper we discover a global generic set in 
$\A_\Delta$ whose special fiber gives the generic Newton polygon 
at almost all prime $p$. 

\begin{theorem}
\label{T:1}
Fix $d_1,d_2\geq 3$ that defines $\Delta$.
For each nontrivial residue $r\bmod \lcm(d_1,d_2)$, 
there exists an explicit piece-wise linear lower convex function  
$\GNP_L(\Delta,t)$ on the interval $[0,2d_1d_2]$, and 
a Zariski dense open subset $\cU_\Delta$ in $\A_\Delta$ defined over $\Q$,
such that for all $f\in \cU_\Delta(\bar\Q)$ and for $p\equiv r\bmod \lcm(d_1,d_2)$ large enough 
$$
\NP_q(L^*_f(\chi_1, s)^{-1})=\GNP_L(\Delta,p).
$$
\end{theorem}

Description of $\GNP_L(\Delta,t)$ can be found in Definition \ref{D:GNP}.
In this paper we use $\coprod$ to denote disjoint union of slopes 
with multiplicity in exponents.
Given a set of rational numbers $\GNP_L(\Delta,t)=\coprod_{j=1}^{2d_1d_2}\alpha_j$,
its {\em $t^m$-adic weighted arithmetic progression} for two-variable case is
the following set
$$
\GNP_L(\Delta,m,t):=\coprod_{i=1}^{t^{m-1}-1}\coprod_{j=1}^{2d_1d_2}
(\frac{i-1+\alpha_j}{t^{m-1}})^{N(i)} 
$$  
where $N(i)=i$ for $1\leq i\leq t^{m-1}$ and $N(i)=2t^{m-1}-i$
for $t^{m-1}<i\leq 2t^{m-1}-1$. 

\begin{theorem}\label{T:arithmeticProg}
Let notation be as in Theorem \ref{T:1}.
If the slopes in $\GNP_L(\Delta,p)$ 
are $\coprod_{j=1}^{2d_1d_2}\alpha_j$,
then for any $f\in \cU_\Delta(\bar\Q)$ and $p\equiv r\bmod \lcm(d_1,d_2)$ 
large enough
$$
\NP_q(L_f^*(\chi_m,s)^{-1})
=\GNP_L(\Delta,m,p)=
\coprod_{i=1}^{p^{m-1}-1}\coprod_{j=1}^{2d_1d_2}
(\frac{i-1+\alpha_j}{p^{m-1}})^{N(i)} 
$$  
where $N(i)=i$ for $1\leq i\leq p^{m-1}$ and $N(i)=2p^{m-1}-i$
for $p^{m-1}<i\leq 2p^{m-1}-1$.
\end{theorem}

In the recent paper \cite{DWX16}, zeta function of Artin-Shreier-Witt covers 
in one variable is studied, that is, $f(x)$ is a 1-variable polynomial.
To some extent, the above theorem is 
a two-dimensional generalization of the main theorem in \cite[Theorem 1.2]{DWX16}, 
answering Question 1.8 of \cite{DWX16}. 
There are some noted difference between our 
2-variable result and their 1-variable one that we shall point out here. 
While Davis-Wan-Xiao proves that 
for any fixed prime $p$ the $p$-adic Newton slopes 
form an arithmetic progression for all $m$ large enough, 
our result fixes any positive integer $m$, and 
proves an arithmetic progress for all prime $p$ large enough. 
Their result applies to all 1-variable polynomials, while 
our result is for generic 2-variable polynomials.
Finally, their 1-variable case is arithmetic progression 
while for 2-variable case it is weighted arithmetic progression. 
Theorem \ref{T:1} also yields a direct generalization 
of the main theorems in \cite{Zh03} and \cite{Zh04}.

Artin-Schreier-Witt cover $A_m(f)$
has the following zeta function
$$
Z(A_m(f),s)=\exp(\sum_{k\ge 1}\# A_m(f)(\F_{q^k})\frac{s^k}{k}).
$$
If one writes it as 
$$Z(A_m(f),s)=\frac{1}{(1-qs)\prod_{k=1}^{m}\prod_{\chi}L_f^*(\chi,s)^{-1}}$$
where the second product ranges over all characters of
conductor equal to $p^k$, it is not hard to 
derive an immediate application of Theorems \ref{T:1} and 
\ref{T:arithmeticProg}:

\begin{corollary}\label{C:1}
Let $f\in\cU_\Delta(\bar\Q)$ and let $p$ be large enough, then 
$$
\NP_q((1-qs)Z(A_m(f),s)^{-1}) = \coprod_{k=1}^{m} \GNP_L(\Delta,k,p)^{p^{k-1}(p-1)}.
$$
\end{corollary}

Let $\cA_f$ be the eigencurve associated to the Artin-Schreier-Witt tower 
defined in \cite[Theorem 1.5]{DWX16}.
For 1-variable polynomial $f$, it was shown by Davis-Wan-Xiao that 
when restricted to the boundary of the weight space,
the eigencurve $\cA_f$ is an infinite disjoint union of subspaces which are finite 
and flat over the weight annulus; in particular, its Newton slopes 
are in arithmetic progressions. This property clearly mirrors that of Coleman-Mazur's eigencurve \cite{CM98}. 
See also \cite{BG15} for 
survey of other related recent progress and conjectures. 
Yet there is  no available explicit conjecture upon 
similar properties for eigencurve over 2-dimensional case.

\begin{theorem}
\label{T:eigen}
Let notation be as in Theorem \ref{T:1}.
For $f\in \cU_\Delta(\bar\Q)$ and $p$ large enough, 
the corresponding eigencurve $\cA_f$ at $p$ is an infinite disjoint union 
of finite flat covers over the weight space with degree $1$ and 
$(2i-1)d_1d_2+1+\delta_\Delta$ 
for some $0\leq \delta_\Delta\leq d_1d_2-1$ for all $i\geq 1$.
\end{theorem}

Research in 1-variable case has witnessed numerous progresses recently:
Ouyang and Yang proved in \cite{OY15} the one-variable case of 
Theorem \ref{T:1}  for $f(x)=x^d+ax$. 
See also \cite{OZ16} for result by Ouyang and Zhang on the family 
$f(x)=x^d+ax^{d-1}$ in a special congruence class of $p$.
Ren-Wan-Xiao-Yu \cite{RWXY} has given analogous 
generalization for higher rank Artin-Schreier-Witt towers, that is, when 
$\Z_p$-extension is generalized to $\Z_{p^\ell}$-extension;
Kosters-Wan \cite{KW16} has formulated an explicit 
class field theory's Schmidt-Witt symbol and a genus bound  for an
Artin-Schreier-Witt towers. On a parallel line of research, recent progresses
made on the Newton slopes of eigencurves in 
\cite{BP15},\cite{BP16},\cite{LWX14}, and \cite{WXZ14} have 
also enriched our understanding of eigencurves in 
our Artin-Schreier-Witt setting. 

Our paper is organized as follows:
Section 2 recalls basic notations of Dwork-Liu-Wan's $T$-adic $L$-functions
and produces a $p$- and $T$-adic interpolation of its Newton polygon for 
generic polynomial in two variables in $\A_\Delta(\bar\F_p)$.
As applications of these interpolations,
we go on to prove the main results of this paper, Theorems \ref{T:1}
and \ref{T:arithmeticProg} 
are proved in Theorem \ref{T:6}. Application to eigencurves relative to Artin-Schreier-Witt covers is discussed in 
Section 4, where Theorem \ref{T:eigen} is proved.

\section{Interpolating Newton slopes of character power series}

We first recall $L$-functions $L_f^*(-,s)$ 
and character power series 
 $C_f^*(-,s)$ (among other things) 
for any $n$-variable polynomial function $f$.
In Section 2.2, we investigate 
polynomial functions (which will turn out to be coefficients of 
$C_f^*$ later) and split it into a product of $p$-adic part
and $T$-adic part, a key step for our interpolation of $C_f^*$.
We apply this finding in $p$-adic and $T$-adic 
interpolation of Newton slopes of $C_f^*$ and $L_f^*$ in Section 2.3.

\subsection{Characteristic power series}
\label{SS:Dwork}

For any positive integer $k$ the exponential sum of a function 
$f(\bx)\in \F_q[x_1,\ldots,x_n,\frac{1}{x_1\cdots x_n}]$
over the $n$-dimensional toric  $\Gm^n$ (where $q=p^a$) is 
$$
S_f^*(k,T):=\sum_{\bx\in \Gm^n(\F_{q^k})} 
(1+T)^{\Tr_{\Q_{q^k}/\Q_p}(\hat{f}(\hat\bx))}
$$
in $\Z_p[[T]]$, where
$\hat{f}$ is the coefficient-wise Teichm\"uller lift of $f$.
The $T$-adic $L^*$-function of $f$ is 
\begin{eqnarray*}
L_f^*(T,s)=\exp(\sum_{k=0}^{\infty}S_f^*(k,T)\frac{s^k}{k})
=\prod_{\bx\in \Gm^n(\F_{q^k})}\frac{1}{1-(1+T)^{\Tr_{\Q_{q^k}/\Q_p}(\hat{f}(\hat{\bx}))}s^{\deg(\bx)}}
\end{eqnarray*}
in $1+s\Z_p[[T]][[s]]$.
The $T$-adic {\em characteristic power series} of $f$ over $\Gm^n$ is the 
generating function
\begin{equation}
\label{E:Cf}
C_f^*(T,s) =\exp(-\frac{1}{(q^k-1)^n}S_f^*(k,T)\frac{s^k}{k})
\end{equation}
in $1+s\Z_p[[T]][[s]]$.
We may verify the following straightforward relations of these above two power series
\begin{eqnarray*}
L_f^*(T,s)^{(-1)^{n-1}} 
&=& 
\prod_{i=0,\ldots,n} C_f^*(T,q^is)^{(-1)^i\binom{n}{i}}
=\frac{\prod_{i=0,2,\cdots} C_f^*(T,q^is)^{\binom{n}{i}}}
{\prod_{i=1,3,\cdots} C_f^*(T,q^is)^{\binom{n}{i}}}
\\
C_f^*(T,s) &=& \prod_{j=0}^{\infty} L_f^*(T,q^js)^{(-1)^{n-1}\binom{n-1+j}{j}}
\end{eqnarray*}
In particular, when $f$ has two variables
then 
\begin{equation}\label{E:relation1}
\left\{
\begin{array}{lll}
L_f^*(T,s)^{-1} &=&  \frac{C_f^*(T,s) C_f^*(T,q^2s)}{C_f^*(T,qs)^2}
\\
C_f^*(T,s) &=& L_f^*(T,s)^{-1} L_f^*(T,qs)^{-2} L_f^*(T,q^2s)^{-3}\cdots.
\end{array}
\right.
\end{equation}
For any $\chi:\Z_p\rightarrow \C_p^*$ with $\chi(1)=\zeta_{p^{m_\chi}}$ 
and for $\pi_\chi=\chi(1)-1$, 
the specialization of these functions at $T=\pi_\chi$ gives rise to $L_f^*(\chi,s)$ and $C_f^*(\chi,s)$, that is, 
\begin{eqnarray}
\label{E:relation3}
L_f^*(\chi,s)=L_f^*(T,s)|_{T=\pi_\chi}, \quad C_f^*(\chi,s)=C_f^*(T,s)|_{T=\pi_\chi}.
\end{eqnarray}

Let $\Delta=\Delta(f)$ be the Newton polytope of $f$, 
namely, 
the convex hull of all $\bv$ with $a_{\bv}\neq 0$ (in the base field)
(for $f=\sum_{\bv} a_\bv \bx^\bv$) and the origin $\vec{0}$ in 
$\R^n$.
Let $S(\Delta)$ be the set of all integral lattice points in the real cone of 
$\Delta$ in $\R^n$. 
For any $\bv\in S(\Delta)$ 
its {\em weight} $w(\bv)$ is defined as the smallest positive rational number $c$ such that 
$\bv\in c\Delta$ if $c$ exists, otherwise $w(\bv)=0$. 
Let $D$ be the denominator of $\Delta$.
(We remark that in our entire paper we fix 
$D=d_1d_2$, slightly different from other custom in the literature where $D$ is the least common multiple of $d_1,d_2$). For any integer $k\geq 0$ let 
\begin{equation}
\label{E:W&H}
\left\{
\begin{array}{lll}
W_\Delta(k)&:=&\#\{\bv\in S(\Delta)|w(\bv)=\frac{k}{D}\}\\
H_\Delta(k)&:=&\sum_{i=0}^{n}(-1)^i\binom{n}{i}W_\Delta(k-i\cdot D).
\end{array}
\right.
\end{equation}
Notice that $H_\Delta(k)\in \Z_{\geq 0}$, 
$H_\Delta(k)=0$ for all $k>n\cdot D$, 
and 
$
\sum_{k=0}^{nD}H_\Delta(k) = n! V(\Delta) 
$
where $V(\Delta)$ is the volume of the polytope $\Delta$ in $\R^n$. 
We also denote 
$\NP_T(-)$ as the $T$-adic Newton polygon with valuation ring obvious from context.
Let 
\begin{eqnarray}
\label{E:HPC}
\HP_C(\Delta)
&:=&\NP_T(\prod_{k=0}^{\infty}(1-
(T^{\frac{k}{D}}s)^{W_\Delta(k)} )).
\end{eqnarray}
It has vertices at 
$\coprod_{i\geq 0}(\sum_{k=0}^{i}W_\Delta(k), \sum_{k=0}^{i} \frac{k\cdot W_\Delta(k)}{D})
$ and 
its slopes with multiplicity are
$\coprod_{i\geq 0}(\frac{i}{D})^{W_\Delta(i)}$. 
Let 
\begin{eqnarray}
\label{E:HPL}
\HP_L(\Delta) &:=&\NP_T(\prod_{k=0}^{nD}(1-(T^{\frac{k}{D}}s)^{H_\Delta(k)} )).
\end{eqnarray}
Its vertices are at 
$\coprod_{i=0}^{n D}(\sum_{k=0}^{i}H_\Delta(k), \sum_{k=0}^{i} \frac{k\cdot H_\Delta(k)}{D})
$.

We denote $\geq$ for the partial order {\em lies above or equal to} 
in the set of lower convex Newton polygons.
Adolphson-Sperber \cite{AS89} and Wan \cite{Wan93}
have shown that for $p\equiv 1\bmod D^*(\Delta)$ 
for some constant $D^*(\Delta)$ (depending only on $\Delta$) 
the Newton polygon of $L_f^*(\chi,s)$ (for $\chi$ of conductor $p$) 
coincides with the Hodge polygon (see $\HP_L(\Delta)$ in (\ref{E:HPL})).
In this paper we shall take for granted of the following facts 
that are due to Liu-Wan (see \cite{LW09}):
\begin{enumerate}
\item  $L_f^*(T,s)^{(-1)^{n-1}}$ 
and $C_f^*(T,s)$ lie in $\Z_p[[T]][[s]]$.
\item 
$\NP_{T^{a(p-1)}}L_f^*(T,s)^{(-1)^{n-1}}\geq \HP_L(\Delta)$ 
in its first $n!V(\Delta)$ terms.
\item 
$\NP_{T^{a(p-1)}}C_f^*(T,s)\geq \HP_C(\Delta)$.
\item For any nontrivial $\chi: \Zp\rightarrow \C_p^*$ we have
$\NP_{\pi_\chi^{a(p-1)}}C_f^*(\chi,s)\geq \HP_C(\Delta)$.
\end{enumerate}

\subsection{Interpolating character series}

For the rest of the paper we 
fix two integers $d_1,d_2\geq 3$ that defines the Newton polytope $\Delta$ in $\R^2$:
namely $\Delta$ is the region $0\leq x\leq d_1$ and $0\leq y\leq d_2$. 
Let $S(\Delta)$ be the set of all integral points
in the cone of $\Delta$. For some subset $\cS$ of $S(\Delta)$, there is 
the {\em simplex partition} $\cS=\cS_1\coprod \cS_2$ 
according to $\frac{v_1}{d_1}\ge \frac{v_2}{d_2}$ or not, see 
Figure \ref{Fig:simplex}. For any $\bv=(v_1,v_2)\in S(\Delta)$, observe its weight
$w(\bv)=\max(\frac{v_1}{d_1},\frac{v_2}{d_2})$, in other words $w(\bv)=\frac{v_i}{d_i}$
if $\bv\in\cS_i$. 

\begin{figure}[h!]
\begin{tikzpicture}[scale=0.6]

\filldraw[fill=gray!30] (0,0) -- (4,0) -- (4,3) -- (0,3) -- (0,0);
\filldraw[fill=gray!50] (0,0) -- (4,0) -- (4,3);
\draw[->] (0,0) -- (7,0);
\draw[->](0,0) --(0,5);
\draw[dashed,thick] (0,0) -- (5,3.7);
\draw[thick] (5,0) -- (5,3.75);
\draw[thick] (0,3.75) -- (5,3.75);

\node[below left] at (0,0) {0};
\node[left] at (0,4) {$d_2$};
\node[below] at (5,0) {$d_1$} ;
\node[] at (1,2) {$\cS_2$};
\node[] at (2.5,1) {$\cS_1$};
\node[below] at (7,0) {$x_1$};
\node[left] at (0,5) {$x_2$};
\filldraw[fill=black] (0,0) circle (3pt);
\filldraw[fill=black] (4,0) circle (3pt);
\filldraw[fill=black] (0,3) circle (3pt);
\filldraw[fill=black] (4,3) circle (3pt);
\filldraw[fill=black] (5,0) circle (2pt);
\filldraw[fill=black] (0,3.75) circle (2pt);
\end{tikzpicture}

\caption{Simplex partition 
$\cS=\cS_1\coprod \cS_2$}
\label{Fig:simplex}
\end{figure}

\noindent Write $D=d_1d_2$, let 
\begin{eqnarray}
\label{E:I_D}
I_D:=\{n\in\Z| 0\leq n< D, d_1|n \mbox{ or } d_2|n\}.
\end{eqnarray}
For any $n\geq 0$, let
\begin{eqnarray}
\label{E:kn}
\cF_n:=\{\bv\in\Z_{\geq 0}^2| w(\bv)\leq \frac{n}{D}\},
\end{eqnarray}
hence $k_n:=\#\cF_n=\sum_{k=0}^{n}W_\Delta(k)$ (as defined in (\ref{E:W&H})).
There is a natural increasing filtration 
$\{\vec{0}\}=\cF_0\subseteq \cdots \subseteq \cF_{n-1}\subseteq \cF_n\subseteq \cdots$
and $\cF_{n-1}\neq \cF_n$ if and only if $n\in I_D$.
For our entire investigation in $2$-variable case we shall
fix a 2-dimensional {\em residue class} $\cR$ of $p\bmod (d_1,d_2)$, that is, 
a pair $(R_1,R_2)$ in $\Z/d_1\Z\oplus \Z/d_2\Z$ such that 
$(R_1,R_2)=(p\bmod d_1,p\bmod d_2)$ with $0\leq R_1\leq d_1-1, 0\leq R_2\leq d_2-1$. Of course each residue is in bijection with 
the residue $\bmod \lcm(d_1,d_2)$.
We denote $p\in \cR$.
For any $\cS\subset S(\Delta)$ let $\Sym(\cS)$ denote the set of all permutations of $\cS$.
For any $\bi\in \cS$ and $\sigma\in \Sym(\cS)$ let 
$\vec{r}_\sigma=(r_1,r_2):=(p\bi-\sigma(\bi))\bmod (d_1,d_2))$ such that 
$0\leq r_1\leq d_1-1, 0\leq r_2\leq d_2-1$. 
Write 
\begin{equation}\label{E:MW}
\left\{
\begin{array}{ll}
M(\cS,\sigma)&=|\cS| - (\frac{1}{d_1}\sum_{\bi\in\cS_1} r_1+\frac{1}{d_2}\sum_{\bi\in \cS_2}r_2)\\
W(\cS,\sigma)&=\sum_{\bi\in \cS}w(\bi)+\frac{1}{p-1}M(\cS,\sigma).
\end{array}
\right.
\end{equation}
Let
$M_{\cS}^0=\min_{\sigma\in \Sym(\cS)}M(\cS,\sigma)$.
Let $\cK_k$ be the set of all subsets 
$\cS\subset S(\Delta)$ with $k$ elements, and let 
$\cK_k^0$ be its subset 
with $\bi$'s with $w(\bi)\leq w(\bi')$ for any 
$\bi'$ not in $\cS$.
Then for $p$ large enough 
$\min_{\cS,\sigma}W(\cS,\sigma)$ is achieved at and only at  $\cS\in \cK_k^0$ for suitable $k$.
We denote by $\Sym^\ell(\cS)$ for all $\sigma$ 
such that $M(\cS,\sigma)=M_{\cS}^0+\frac{\ell}{D}$.
We write 
\begin{equation*}
\boxed{
W_k^\ell=
\sum_{\bi\in\cS}w(\bi) +\frac{M_\cS^\ell}{p-1};\qquad
M_\cS^\ell = M_\cS^0+\frac{\ell}{D}
}
\end{equation*}
for some $\cS\in \cK_k^0$.
For every $\sigma\in \Sym(\cS)$, define a rational number
\begin{eqnarray*}
Q_{\cS_2,\sigma}&:=&
\prod_{\bi\in\cS_2} \frac{1}{\pfloor{\frac{pi_1-\sigma(i_1)}{d_1}}!
(\pfloor{\frac{pi_2-\sigma(i_2)}{d_2}}-\pfloor{\frac{pi_1-\sigma(i_1)}{d_1}})!}. 
\end{eqnarray*}
Let $Q_{\cS_1,\sigma}$ be defined similarly by swapping subindices $1$ and $2$ in
the above formula. 
Let $a_{d_1,d_2}=a_{d_1,0}=a_{0,d_2}=a_{0,0}=1$ and let $n\in I_D$.
For any $\ell\geq 0$ define a polynomial function 
in $(\Zp\cap\Q)[a_\bv]$
\begin{eqnarray*}
G_{k_n,p}^\ell
&=&
\sum_{\sigma\in \Sym^\ell(\cF_n)}
(-1)^{k_n}\sgn(\sigma)Q_{(\cF_n)_1,\sigma} Q_{(\cF_n)_2,\sigma}\cdot \prod_{\bi\in\cF_n}a_{\vec{r}_\sigma}.
\end{eqnarray*}
Then we define a polynomial in $\Z[a_\bv]$:
\begin{eqnarray}
\label{E:GknR}
\boxed{
G_{k_n,\cR}^\ell=\sum_{\sigma\in \Sym^\ell(\cF_n)}
(-1)^{k_n}\sgn(\sigma)U_{k_n,p} Q_{(\cF_n)_1,\sigma} Q_{(\cF_n)_2,\sigma}\prod_{\bi\in\cF_n}
a_{\vec{r}_\sigma}}
\end{eqnarray}
where 
\begin{eqnarray*}
U_{k_n,p} :=\prod_{\bi\in (\cF_n)_2} \pfloor{\frac{pi_1}{d_1}}!\pfloor{p(\frac{i_2}{d_2}-\frac{i_1}{d_1})}!\prod_{\bi\in (\cF_n)_1} \pfloor{\frac{pi_2}{d_2}}!\pfloor{p(\frac{i_1}{d_1}-\frac{i_2}{d_2})}!,
\end{eqnarray*}
We shall use $O(-)$ to denote higher order terms $p$-adically 
or $\pi$-adically below, which should be clear from context.

\begin{lemma}
\label{L:open}
Fix $d_1,d_2\geq 3$, and fix a nontrivial
residue class $\cR$ in $\Z/d_1\oplus \Z/d_2$. Fix $n\in I_D$.
Then $G_{k_n,\cR}^\ell=U_{k_n,p} G_{k_n,p}^\ell+O(p^{>1})$.
There exists a least nonnegative integer $\ell<D^2$
such that  $G_{k_n,\cR}^\ell$ is a nonzero polynomial 
in $\Z[a_\bv]$.
\end{lemma}

\begin{proof}
\begin{enumerate}
\item 
After some elementary cancellation, we notice by its very definition 
that $(U_{k_n,p} Q_{(\cF_n)_1,\sigma} Q_{(\cF_n)_2,\sigma}\bmod p)$ is independent of $p$.
Moreover the subindices $\vec{r}_\sigma$ are also independent of $p$.
This proves that $G_{k_n,\cR}^\ell$ depends only on $d_1,d_2,\Sym^0(\cF_n)$ and $\cR$.

\item 
Notice that for each permutation $\sigma\in\Sym(\cF_n)$,
there is a corresponding monomial $\prod_{\bi\in\cF_n}a_{\vec{r}_\sigma}$ in the sum of 
$G_{k_n,\cR}$ by 
$$\tau:
\sigma\mapsto \prod_{\bi\in\cF_n}
a_{\vec{r}_\sigma}$$ where $\vec{r}_\sigma=\bar{p\bi-\sigma(\bi)}$.
For nonzero-ness of the polynomial $G_{k_n,\cR}^\ell$
it suffices to find a monomial $\prod_{\bi\in\cF_n}a_{\vec{r}_\sigma}$ that is unique (namely, 
it can not be cancelled by other terms).
To this end, we shall explicitly locate $\sigma$ 
such that its corresponding monomial $\tau(\sigma)$ is unique. 

Some preparation is in order:
First of all, we define a total order $\Xi$ on the residue 
set $\Z/d_1\Z\oplus \Z/d_2\Z$ as follows: 
first of all we order them in terms of weight (as defined in Section \ref{SS:Dwork}), 
namely,  define $\bv\geq \bv'$ if $w(\bv)\geq w(\bv')$; for the subset of 
vectors with equal weight, we order them in lexicographic order
with the first coordinate of higher order. This defines a total order,
we write $\bv>\bv'$.
This order $\Xi$ defines a total order on the set of monomials 
$\prod_{\br\in\Z/d_1\oplus \Z/d_2}a_\br^{t_\br}$ (index $\br$ is ordered by $\Xi$) as follows:
$\prod_{\br} a_\br^{t_\br}> \prod_{\br}a_\br^{s_\br}$
when $(t_\br)_\br > (s_\br)_\br$ in lexicographic order. By abuse of notation we call this 
total order on $\Xi$ as well. We shall locate the permutation $\sigma\in \Sym(\cF_n)$ that corresponds to the monomial 
of the {\em unique} highest $\Xi$-order. 
Let $[R]$ be the $k_n\times k_n$ residue matrix $\bar{p\bi-\bj}$ with $\bi,\bj\in\cF_n$. 
Then each row (and column) consists of pairwise distinct vectors 
in $\Z/d_1\Z\oplus \Z/d_2\Z$.
Take all entries of highest $\Xi$-order in the two diagonal blocks 
of $[R]$ where $\bi,\bj\in (\cF_n)_1$ or $\bi,\bj\in (\cF_n)_2$.
By definition,  any two entries of the same $\Xi$-order necessarily lie in different rows 
and different columns. Collect every $(\bx,\by)$ in $[R]$,
assign $\sigma(\bx)=\by$,  then cross off the $\bx$-the row and $y$-th column. 
Repeat this step on the left-over matrix until entire matrix $[R]$ is crossed off. 
This obviously defines a permutation in $\Sym(\cF_n)$.
Then its corresponding monomial under $\tau$ is 
$$\tau(\sigma)
=
\prod_{\br\in\Z/d_1\oplus \Z/d_2} a_{\br}^{m_\br}
$$ 
where $m_\br=\#\{\bi\in\cF_n| \bar{p\bi-\sigma(\bi)} = \br\}$.
By construction it is of the highest $\Xi$-order in the two diagonal 
blocks, hence it is unique.

\item
Each permutation $\sigma$ constructed from part 2) above yields a specialization 
$$M(\cF_n,\sigma)=k_n-(\frac{1}{d_1}\sum_{(\cF_n)_1} r_1 + \frac{1}{d_2}\sum_{(\cF_n)_2} r_2)
$$
where the two sums ranges over all entries $(r_1,r_2)$ picked in 
the residue matrix $[R]$ from the above algorithm.
Suppose the minimum $M_{\cF_n}^0$ is achieved at 
$\sigma^0$, and its corresponding entries in $[R]$ are
$(r_1^0,r_2^0)$ at the same same row of $[R]$. 
Then $|r_1-r_1^0|<d_1$ and 
$|r_2-r_2^0|<d_2$, and hence 
$M(\cF_n,\sigma)-M_{\cF_n}^0 < D$.
This proves that $M(\cF_n,\sigma)<M_{\cF_n}^0+\frac{D^2}{D}$
and hence $\sigma\in \Sym^\ell(\cF_n)$ for some $\ell<D^2$ 
by definition (\ref{E:GknR}). We thus conclude that 
$G_{k_n,\cR}^\ell$ is nonzero.
\end{enumerate}
\end{proof}

\subsection{Generic Newton slope formula}

Let notation be as in Lemma \ref{L:open}.
For simplicity, from now on we shall denote 
by $G_{k_n,\cR}:=G_{k_n,\cR}^\ell$, 
$M_\cS=M_\cS^\ell$,
and $W_k=W_k^\ell$.
Consider $W_k$ as a function of
$k\in\Z_{\geq 0}$. Define 
a slope function $s(k,k'):=\frac{W_k-W_{k'}}{k-k'}$ for any $k\neq k'$ such that
$0\leq k,k'\leq D$.
This is just the slope of the line connecting 
$(k,W_k)$ and $(k',W_{k'})$,
the defining points for $\gNP$.
Suppose $k_n\leq k< k_{n+1}$ for some $n\in I_D$, since $\cS\in\cK_k^0$
we have $\cS=\cF_n\coprod (\cS\cap \cF_{n+1})$.
So 
$$
W_k-W_{k_n}
=\sum_{\bj\in\cS\cap \cF_{n+1}}w(\bj)+\frac{M_\cS-M_{\cF_n}}{p-1}
=\frac{n+1}{D}(k-k_n)+\frac{M_\cS-M_{\cF_n}}{p-1}.
$$
If $k\neq k_n$ then 
\begin{eqnarray}
\label{E:s}
s(k,k_n)
=\frac{n+1}{D}+\frac{1}{p-1}\frac{M_\cS-M_{\cF_n}}{k-k_n}.
\end{eqnarray}

\begin{definition}
\label{D:sNP}
	Let $d_1,d_2\ge 3$, and $\cR$ be given.
	For any $n\in I_D$, let $\varepsilon_n:=M_{\cF_n}-M_{\cF_m}$ where
	$n,m\in I_D$ and $k_m=k_{n-1}$; 
	and write $\beta_n(t):=\frac{n}{D}+\frac{\varepsilon_n}{t-1}$.
Define a piece-wise linear function that is parameterized by $t$:
\begin{eqnarray*}
\sNP(\Delta,t)&:=&\NP_T(\prod_{n\in I_D}(1-(T^{\beta_n(t)}s)^{W_\Delta(n)})).
\end{eqnarray*}
Consider $\sNP(\Delta,t)$ a function in $t$ with domain consisting of all prime number $p$ in the residue class $\cR$ only. 
Then $\sNP(\Delta,p)$ is the lower convex hull of the points 
$(k_n,W_{k_n})$. 
Its slopes are $\coprod_{n\in I_D}\beta_n(p)^{W_\Delta(n)}$ and $0\leq \beta_n(p)
\leq 1$.
Let $\gNP(\Delta,p)$ be the lower convex hull of the points $(k,W_k)$
for all $0\leq k\leq D$ in $\R^2$.
\end{definition}

\begin{definition}[Formula for generic Newton polygon: classical character]
\label{D:GNP}
Fix integers $d_1,d_2\geq 3$ that defines $\Delta$.
Fix a nontrivial residue class $\cR$ in $\Z/d_1\oplus \Z/d_2$.
For any $t$ ranging over all primes $p$ in $\cR$ define a $t$-parameterized 
piece-wise linear lower convex function on $[0,2d_1d_2]$
$$
\GNP_L(\Delta,t)=
\NP_T\prod_{n\in I_D}\left(1-
(T^{\beta_n(t)}s)^{W_\Delta(n)} 
\right)\left(1-(T^{2-\beta_n(t)}s)^{W_\Delta(n)} 
\right)
$$
where $\varepsilon_n=M_{\cF_n}-M_{\cF_{n-1}}$ and  
$\beta_n(t)=\frac{n}{D}+\frac{\varepsilon_n}{t-1}$.
Let $\GNP_C(\Delta,t)$ be the product of $\GNP_L(\Delta,t)$
and the algebraic polygon of $\frac{1}{(1-s)^2}$.
\end{definition}

\begin{remark}\label{Remark:vertices}
(1) 
Slopes of 
$\GNP_C(\Delta,p)$ are 
$$
\boxed{
\GNP_C(\Delta,p)=\coprod_{i=1}^{\infty} 
\coprod_{n\in I_D} \{i-1+\beta_n(p), i+1-\beta_n(p)\}^{i\; W_\Delta(n)}
}
$$
where $0\leq \beta_n(p)\leq 1$. 
Since $\cR$ is nontrivial, for $p$ large enough, we have
$i-1+\beta_n(p)\neq  i+1-\beta_m(p)$ where $n\neq m\in I_D$ and $i\geq 0$
unless and only unless $\varepsilon_n=\varepsilon_{2D-n}=0$.

(2) 
These Newton polygons are defined in such a way that they depend on 
$\Delta$ and the residue class $\cR$. Moreover,
\begin{eqnarray*}
\lim_{t\rightarrow\infty}\GNP_L(\Delta,t)=\HP_L(\Delta),\quad \lim_{t\rightarrow\infty}\GNP_C(\Delta,t)=\HP_C(\Delta).
\end{eqnarray*}
\end{remark}

\begin{proposition}
\label{P:2}
\begin{enumerate}
\item
Let $\HP_C(\Delta)$ be as in (\ref{E:HPC}), 
we write $\NP^{[D]}$ for the truncation of first $D$ terms of a Newton polygon $\NP$. 
Then $$\sNP(\Delta,p)\geq \gNP(\Delta,p)\geq \HP_C(\Delta)^{[D]}.$$ 
\item 
The two Newton polygons 
$\gNP(\Delta,p)$ and $\sNP(\Delta,p)$ both have horizontal length equal to $D$.
For $p$ large enough (depending only on $d_1,d_2$),
$\gNP(\Delta,p)=\sNP(\Delta,p)$ with vertices exactly at $(k_n,W_{k_n})$. Its
slopes with multiplicities are $\coprod_{n\in I_D}\beta_n^{W_\Delta(n)}$.
Moreover, all slopes are in $[0,1]$.

\item 
For $p$ large enough (depending only on $d_1,d_2$), 
the two Newton polygons $\sNP(\Delta,p)$ and $\HP_C(\Delta)^{[D]}$ 
have vertices both at $x=k_n$ ($n\in I_D$) where the $y$-coordinate gap 
is equal to $\frac{\varepsilon_n}{p-1}$, where 
$\varepsilon_n\in \frac{\Z_{\geq 0}}{\lcm(d_1,d_2)}$.
Furthermore, 
$$\lim_{p\rightarrow \infty} \gNP(\Delta,p) = \lim_{p\rightarrow \infty}\sNP(\Delta,p)=\HP_C(\Delta)^{[D]}.$$
\end{enumerate}
\end{proposition}

\begin{proof}
1) That $\sNP(\Delta,p)\geq \gNP(\Delta,p)\geq \HP_C(\Delta)$ for all $p$ is immediate by their definitions.

\noindent 2) We shall show  that $(k_n,W_{k_n})$ is a vertex point for $\gNP(\Delta,p)$ for each 
$n\in I_D$.
Pick two positive integers $k,k'$ such that $k<k_n<k'$. Notices that
$k'-k_n\geq 1$ and $1\leq k_n-k\leq D$.
So 
\begin{eqnarray}\label{E:k}
\frac{1}{k'-k_n}-\frac{1}{k_n-k}
& \le &
1-\frac{1}{k_n}\le 1-\frac{1}{D}.
\end{eqnarray}
By definition we have $M_\cS< D$, and hence 
\begin{eqnarray}
\label{E:MS}
|M_\cS-M_{\cF_n}|\le \max(M_\cS,M_{\cF_n})
\leq 2D.
\end{eqnarray}
For any $k<k_n$, by the argument above and (\ref{E:s}), we have
$$
s(k,k_n) \leq \frac{n}{D}+\frac{1}{p-1}\frac{D}{k_n-k}.
$$
Since $k_n<k'$ we have $k_m\leq k'<k_{m+1}$ for some $m\geq n+1$ and 
$m\in I_D$, 
we have similarly
$$
s(k_n,k')\geq \frac{n+1}{D}+\frac{1}{p-1}\frac{-D}{k'-k_n}.
$$
Therefore, 
\begin{eqnarray*}
s(k_n,k')-s(k,k_n) 
&\geq &
\frac{1}{D}-\frac{D}{p-1}(\frac{1}{k'-k_n}-\frac{1}{k_n-k})
\\
&\geq &
\frac{1}{D}-\frac{D}{p-1}(1-\frac{1}{D})\\
&\geq &
\frac{1}{D}-\frac{D-1}{p-1}.
\end{eqnarray*}
When $p>D^2-D+1$ we have
that 
$s(k_n,k')-s(k,k_n) >0$.
By the definition of slope function $s(-,-)$,
this means that $(k_n,W_{k_n})$ is a vertex on $\gNP(\Delta,p)$.

For any $k>0$, write $(k,W'_k)$ for points on $\sNP(\Delta,p)$.
We claim that for $p$ large enough,
$$|W'_k-W_k|\leq \frac{2D}{p-1}.
$$
By (\ref{E:s}) and (\ref{E:MS}), if $k\neq k_n$ 
\begin{eqnarray}
\label{E:slope}
|s(k,k_n) - \frac{n+1}{D}|
= 
\frac{1}{p-1}|\frac{M_\cS-M_{\cF_n}}{k-k_n}|
\leq \frac{1}{p-1}\frac{D}{k-k_n}.
\end{eqnarray}
Let $\cS\in\cK_k^0$ and $k_n\leq k<k_{n+1}$.
By definition  and Proposition \ref{P:2} (1),
$W'_k=W_k+\frac{k-k_n}{k_{n+1}-k_n}(W_{k_{n+1}} - W_{k_n})$. 
It follows
\begin{eqnarray*}
W'_k-W_k
&=&
(W'_k-W_{k_n})-(W_k-W_{k_n})\\
&=&
(\frac{W_{k_{n+1}}-W_{k_n}}{k_{n+1}-k_n})(k-k_n)
-(\frac{n+1}{D}(k-k_n)+\frac{M_\cS-M_{\cF_n}}{p-1})\\
&=&
(s(k_{n+1},k_n)-\frac{n+1}{D})(k-k_n)
-\frac{M_\cS-M_{\cF_n}}{p-1}.
\end{eqnarray*}
We have
$$
|(W'_k-W_k) - (s(k_{n+1},k_n)-\frac{n+1}{D})(k-k_n)|
\leq 
\frac{1}{p-1}|M_\cS-M_{\cF_n}|\leq \frac{D}{p-1}.
$$
Combining  the above and (\ref{E:slope}), the difference 
in $y$-coordinates of $\gNP(\Delta,p)$ and $\sNP(\Delta,p)$ at each $k$ is 
$
|W'_k-W_k|\leq \frac{2D}{p-1}
$
which approaches $0$ as $p\rightarrow \infty$.
This proves that $\sNP(\Delta,p)=\gNP(\Delta,p)$ for $p$ large enough.

\noindent 3) The vertices of $\HP_C(\Delta)$ are at
$(k_n, \sum_{\bi\in \cF_n}w(\bi))$, 
so the gap between $\HP_C(\Delta)$ and 
$\gNP(\Delta)$ at $x=k_n$ is $\varepsilon_n/(p-1)$.
But $\varepsilon_n$ only depends on the residue class $\cR$.
Hence it is bounded in terms of $d_1,d_2$. It follows from the definition 
of $\varepsilon_n$ that its denominator is the least common multiple of $d_1,d_2$.
\end{proof}

\section{Existence of global generic polynomials}
 
\subsection{Interpolating character series in 
$p$- and $T$-adically}

For the $p$-adic Artin-Hasse exponential $E(x)$, write it in $x$-adic expansion $E(x)=\sum_{j=0}^\infty u_jx^j$ where $u_j\in\Zp\cap\Q$. 
In particular, for $0\leq j\leq p-1$, $u_j=\frac{1}{j!}$.
Consider $E(x)$ as a function 
$E: \Z_p[[T]]\longrightarrow \Z_p[[T]]$. Let $\pi\in \Z_p[[T]]$ such that
$E(\pi)=1+T$. Then $T=E(\pi)-1=\pi+O(\pi^{>1})$.

We define two polynomials in $(\Z_p\cap\Q)[a_\bv][\pi]$ below:
For any $\bv\in S(\Delta)$ let
\begin{equation}
\label{E:b}
B_\bv:=\sum_{j_{\bw}} (\prod_{\bw\in\Delta\cap\Z^2}
u_{j_\bw}a_\bw^{j_\bw})\pi^{\sum_{\bw}j_\bw}
\end{equation}
where the sum ranges over the set of all $(j_\bw)_{\bw}$ with 
$\sum_{j_\bw} j_{\bw}\bw = \bv$ and $j_\bw\in\Z_{\geq 0}$.
For any $k\geq 1$ define 
\begin{eqnarray}
\label{E:C_k}
H_k:=(-1)^k\sum_{\cS\in \cK_k}\sum_{\sigma\in\Sym(\cS)} \sgn(\sigma)\prod_{\bi\in \cS} B_{p\bi-\sigma(\bi)}
\end{eqnarray}
where $\cS$ ranges over all sets in $S(\Delta)$ of cardinality $k$.

\begin{lemma} 
\label{L:b}
Let $n\in I_D$ and $k_{n-1}<k\leq k_n$. Let $\bi\in \cS\in\cK_k^0$.

(1) Let $\bi\in\cS_1$.
Then for $p$ large enough we have
$$
B_{p\bi-\bj}
=\pi^{t} u_{t_1}u_{t_2} 
a_{\bar{p\bi-\bj}}+O(\pi^{>})
$$
where 
$t=\pfloor{\frac{pi_1-j_1}{d_1}}+1$, 
$t_1=\pfloor{\frac{pi_2-j_2}{d_2}}$,
and 
$t_2=\pfloor{\frac{pi_1-j_1}{d_1}}-\pfloor{\frac{pi_2-j_2}{d_2}}$.

(2) For $\bi\in\cS_2$ same statement holds with subindices swapping.

(3) For $p$ large enough we have
\begin{eqnarray*}
\ord_\pi B_{p\bi-\bj} 
=
\left\{
\begin{array}{ll}
\pfloor{\frac{pi_1-j_1}{d_1}}+1 &\mbox{if $\bi\in\cS_1$}\\
\pfloor{\frac{pi_2-j_2}{d_2}}+1 & \mbox{if $\bi\in\cS_2$}.
\end{array}
\right.
\end{eqnarray*}
\end{lemma}
\begin{proof}
Restrict $\bi,\bj$ in $\Delta$ in the proof.
For $p$ large enough $p\bi-\bj$ lies in wherever $\bi$ is in the simplex partition of 
$S(\Delta)$, without loss of generality we may assume $\bi\in \cS_1$.
Then among all $j_\bw$ such that $\bv =\sum_{j_\bw} j_\bw \bw$,
the following sum
$$
\bv = \pfloor{\frac{v_2}{d_2}}(d_1,d_2) 
+ (\pfloor{\frac{v_1}{d_1}}-\pfloor{\frac{v_2}{d_2}})(d_1,0)+1\cdot (\bar\bv)
$$
gives the unique minimal $\sum_{j_\bw} j_\bw$ that is equal to 
$\pfloor{\frac{v_1}{d_1}}+1$.
Using our hypothesis that $a_{d_1,d_2}=a_{d_1,0}=a_{0,d_2}=1$
we conclude Part (1) and its $\pi$-adic valuation in Part (3).
The other case $\bi\in \cS_2$ in Part (2) is similar which we omit. 
\end{proof}

For ease of notation we suppress $\ell$ in 
all $W_{k_n}$ and $G_{k_n,\cR}$ from Lemma \ref{L:open} below
whenever context allows.

\begin{proposition}\label{P:Hkn}
Let notation be as in Lemma \ref{L:open}.
For any $n\in I_D$, 
\begin{eqnarray*}
H_{k_n}&=& \pi^{(p-1)W_{k_n}}U_{k_n,p}^{-1}
G_{k_n,\cR}+  O(\pi^{p-1})^{>W_{k_n}}+\pi^{(p-1)W_{k_n}}O(p^{>1}).
\end{eqnarray*}
\end{proposition}
\begin{proof}
Let $W(\cF_n,\sigma)$ be as in (\ref{E:MW}).
By Lemma \ref{L:b} and generic facial decomposition from \cite{Zh12},
the $\pi$-adic dominating $\sigma$-term occurs only if 
$\sigma$ stabilizes $(\cF_n)_1$ and $(\cF_n)_2$,
namely $\sigma(\cF_n)_\ell=(\cF_n)_\ell$.
By Lemma \ref{L:open} there exists a least $\ell$ such that 
$G_{k_n,\cR}^\ell$ is nonzero polynomial in
the $\pi$-adic formal expansion of $H_{k_n}$.  
Write $G_{k_n,\cR}=G_{k_n,\cR}^\ell$,
\begin{eqnarray*}
H_{k_n}
&=&
\pi^{(p-1)W_{k_n}}
\sum_{\sigma\in \Sym^\ell(\cS)}
(-1)^{k_n}\sgn(\sigma) Q_{(\cF_n)_1,\sigma}Q_{(\cF_n)_2,\sigma}
\prod_{\bi\in\cF_n} a_{\vec{r}_\sigma}
+O(\pi^{p-1})^{>W_{k_n}}.\nonumber
\\
&=&\pi^{(p-1)W_{k_n}} G_{k_n,p}+ O(\pi^{p-1})^{>W_{k_n}}\\
&=&\pi^{(p-1)W_{k_n}}U_{k_n,p}^{-1} G_{k_n,\cR} + O(\pi^{p-1})^{>W_{k_n}}
+\pi^{(p-1)W_{k_n}}O(p^{>1}),
\end{eqnarray*}
where the last equation follows from Lemma \ref{L:open}.
\end{proof}

Let $f(x_1,x_2)=\sum_{\bv}a_{\bv} \bx^{\bv}$
in $\Q[x_1,x_2]$  where $d_1,d_2\geq 3$. 
Dwork splitting function of $f(x_1,x_2)$ above is defined as
$
E_f(\bx):=\prod_{\bv\in\Delta(f)} E(\pi \ha_{\bv} \bx^\bv)
$
where $\ha_\bv$ is the Teichm\:uller lifting of $a_\bv$
and $\bv\in\Delta(f)\cap\Z^2$.
Then 
\begin{equation}\label{bv}
E_f(\bx)=\sum_{\bv\in S(\Delta)} B_{\bv}(\ha_\bw) \bx^{\vec{v}}
\end{equation}
where 
$B_\bv(\ha_\bw)=B_\bv|_{a_\bw=\ha_\bw}$.

Let  $\psi$ denote the Dwork operator on the $p$-adic Banach space 
\begin{eqnarray*}
\cD&:=&\{\sum_{\bv\in S(\Delta)}\alpha_\bv \bx^\bv\in\bar\Q_p[[\bx]]|
\lim_{|\bv|\rightarrow \infty} p^{\frac{cw(\bv)}{p-1}}|\alpha_\bv|_p = 0
\}
\end{eqnarray*}
for some small enough $c\in \R_{>0}$.
For $\bi\in S(\Delta)$ Dwork operator $\psi=\psi_p\cdot E_f(\bx)$
acts on the monomial basis $\vec{x}^{\bj}$ of $\cD$ as follows
\begin{align*}
	\psi(\vec{x}^{\bj})&=\psi_p(E_f(\vec{x})\cdot \vec{x}^{\bj})
		=\psi_p(\sum_{\vec{v}\in S(\Delta)} B_{\vec{v}}(\ha) \vec{x}^{\vec{v}+\bj})
		=\sum_{\bi\in S(\Delta)} B_{p\bi-\bj}(\ha) \vec{x}^{\bi}.
\end{align*}
With respect to the monomial basis $\{\pi^{w(\bj)}\bx^\bj\}_{\bj\in S(\Delta)}$, 
we have by Dwork theory \cite{Dw62},
\begin{eqnarray*}
C_f^*(\pi,s)
&=&\det(1-\Matrix(\psi) s)
=\det(1-(\pi^{w(\bj)-w(\bi)} B_{p\bi-\bj}))_{\bi,\bj}s)\\
&=&\sum_{k=0}^\infty H_k(\ha_\bv) s^k
\end{eqnarray*}
where $H_k(\ha_\bv)=H_k|_{a_\bv=\ha_\bv}$ as usual.

\begin{proposition}
\label{P:dwork}
(1) Suppose $G_{k_n,\cR}(\ha)\neq 0$
then $\ord_{T^{p-1}}H_{k_n}(\ha)=W_{k_n}$.

\noindent (2)   
Suppose $G_{k_n,\cR}(\ha)\in\bar\Z_p^*$, then for any 
$\pi_\chi\in\bar\Z_p$ with
$\ord_p(\pi_\chi)>0$,
we have 
$\ord_{\pi_{\chi}^{p-1}}H_{k_n}(\ha)=W_{k_n}$.
In particular, the first $D$ terms of the character power series 
specialization at $\pi=\pi_\chi$ are 
\begin{eqnarray*}
C_f^*(\pi_\chi,s)^{[D]}
&=& \sum_{n\in I_D} U_{k_n,p}^{-1} G_{k_n,\cR}(\ha) \pi_\chi^{(p-1)W_{k_n}} s^{k_n} + (\mbox{higher-terms})
\end{eqnarray*}
where higher-terms contains all higher $p$-adic order monomials
in each coefficients of $s^{k_n}$.
\end{proposition}
\begin{proof}
This statement follows directlyly from Proposition \ref{P:Hkn}.
\end{proof}

\subsection{Generic slopes for classical character}

Let $\A$ denote the coefficient space of all polynomials 
$f=\sum_{\bv\in\Delta\cap\Z^2}a_\bv \bx^\bv$ 
with $a_{d_1,d_2}=a_{d_1,0}=a_{0,d_2}=1$.
For $p\equiv (1,1)\bmod (d_1,d_2)$ 
the Newton polygons coincide with Hodge polygons (see \cite{Wan93}), 
we exclude this 
trivial residue classes  for the rest of the paper. 
For any $f$ in $\A(\bar\Q)$ we shall always assume 
its residue field mod $p$ is $\F_q$ for some $q=p^a$ and $a\in\Z_{\geq 1}$.
Let 
\begin{eqnarray*}
G:=\prod_{\cR}\prod_{n\in I_D}G_{k_n,\cR}
\end{eqnarray*}
where $\cR$ ranges over all nontrivial residue classes 
in $\Z/d_1\Z\oplus \Z/d_2\Z$,
and $G_{k_n,\cR}=G_{k_n,\cR}^\ell$ is as in Lemma \ref{L:open}.

\begin{proposition}
\label{P:gNP}
Let $\cU$ be the subset in $\A$ defined by 
$G\cdot \prod_{\bv\in\Delta\cap\Z^2}a_{\bv}\neq 0$.
Let $\cU^o$ be the subset of $\A$ defined by 
$G\neq 0$. 
Then $\cU\subset\cU^o$ are both Zariski dense open subsets of $\A$ defined over $\Q$.

\begin{enumerate}
\item Let $f$ in $\cU(\bar\Q)$. For any character $\chi:\Zp\rightarrow 
\C_p^*$ and for $p$ large enough
$$
\NP_{T^{a(p-1)}}C_f^*(T,s)^{[D]}=\NP_{\pi_\chi^{a(p-1)}}C_f^*(\chi,s)^{[D]} 
=\gNP(\Delta,p).
$$
In particular, if $\chi_1$ has conductor $p$, then 
$\NP_qC_f^*(\chi_1,s)^{[D]}=\gNP(\Delta,p)$.
\item 
Let $f\in \A(\Q)$. Then for all $p$ large enough  
and for any character $\chi:\Zp\rightarrow \C_p^*$
$$
\NP_{T^{p-1}}C_f^*(T,s)^{[D]}=\NP_{\pi_\chi^{p-1}}C_f^*(\chi,s)^{[D]} 
=\gNP(\Delta,p)
$$
if and only if $f\in\cU^0(\Q)$.
In particular, 
$\NP_pC_f^*(\chi_1,s)^{[D]}=\gNP(\Delta,p)$.
\end{enumerate}
\end{proposition}

\begin{proof}
Our fist statement follows from that 
$G$ is nonzero polynomial in $\Q[a_\bv]$ by Lemma \ref{L:open}.

(1) 
By hypothesis, $\prod_{\bv\in\Delta\cap\Z^2}a_{\bv}\ne 0$ 
and hence $\ha_{\bv}\in\Z_p^*\cap \Q$ for all $\bv\in\Delta\cap\Z^2$. 
This reduces to consider $f\in\cU(\Z)$ for $p$ large enough by the 
Transformation Lemma of \cite[Section 5]{Zh12}.
By Lemma \ref{L:open} $G$ is a nonzero polynomial with coefficients in $\Q$, this 
shows that $\cU$ is Zariski dense open subset of $\A$ defined over $\Q$.
Since each $G_{k_n,\cR}$ is a nonzero polynomial with coefficients in $\Q$,
its specialization $G_{k_n,\cR}(a_\bv)$ lies in $\Q$ for all $a_\bv\in\Z$. 
Without loss of generality, fix $\cR$ below.
Let $p\in\cR$ be large enough, $G_{k_n,\cR}(a_{\bv})\in\Zp^*$,
but $G_{k_n,\cR}(\ha_\bv)\equiv G_{k_n,\cR}(a_\bv)\bmod p$,
so $G_{k_n,\cR}(\ha_\bv)\in\Z_p^*$.
But $U_{k,p}\in\Zp^*$, hence for $p$ large enough by Proposition \ref{P:dwork},
$\NP_{T^{p-1}}(C_f^*(T,s)^{[D]})=\NP_{\pi_\chi^{p-1}}C_f^*(\chi,s)^{[D]}
=\sNP(\Delta,p)=\gNP(\Delta,p)$.

(2) This part of the proof is similar to the first part, 
except that we do not need $\prod_{\bv\in\Delta\cap\Z^2}a_\bv\neq 0$ 
in hypothesis as we do not need the aforementioned Transformation Lemma.
\end{proof}

Recall from above that 
\begin{equation*}
\left\{	
\begin{array}{lll}
\GNP_L(\Delta,p)&=&\coprod_{n\in I_D}\{\beta_n(p),(2-\beta_n(p))\}^{W_\Delta(n)}.\\
\GNP_C(\Delta,p)&=&\coprod_{i=1}^{\infty} 
\coprod_{n\in I_D} \{i-1+\beta_n(p), i+1-\beta_n(p)\}^{i\; W_\Delta(n)}.
\end{array}
\right.
\end{equation*}
We shall see below that 
for the classical character $\chi_1$ of conductor $p$, 
the asymptotic generic Newton polygons of $L_f^*$ and $C_f^*$-functions
are given by $\GNP_L(\Delta,p)$ and $\GNP_C(\Delta,p)$, respectively.

\begin{theorem}
\label{T:GNP} Let notation be as in Proposition \ref{P:gNP}.\\
\noindent (1) 
Suppose $f\in\cU(\bar\Q)$ and for $p$ large enough, 
\begin{eqnarray*}
\NP_q(L_f^*(\chi_1,s)^{-1})=
\GNP_L(\Delta,p),\qquad
\NP_q(C_f^*(\chi_1,s)) =
\GNP_C(\Delta,p).	
\end{eqnarray*}

\noindent (2) 
Let $f\in\A(\Q)$. Then for $p$ large enough 
\begin{eqnarray*}
\NP_p(L_f^*(\chi_1,s)^{-1})=
\GNP_L(\Delta,p),\qquad
\NP_p(C_f^*(\chi_1,s)) =
\GNP_C(\Delta,p)	
\end{eqnarray*}
if and only if $f\in \cU^o(\Q)$.
\end{theorem}
\begin{proof}
(1) By (\ref{E:relation1}) and (\ref{E:relation3}), 
$$\NP_q(L_f^*(s)^{-1})^{[D]}=\NP_qC_f^*(s)^{[D]}=\gNP(\Delta,p)=
\coprod_{n\in I_D}\beta_n^{W_\Delta(n)}$$ 
by Proposition \ref{P:gNP} (1).
By symmetry of the Newton polygon $\NP_q(L_f^*(\chi_1,s)^{-1})$,
the rest $D$ slopes are $\coprod_{n\in I_D}(2-\beta_n)^{W_\Delta(n)}$.
In summary $\NP_q(L_f^*(\chi_1,s)^{-1})=\GNP_L(\Delta,p)$. 
Consequently we have $\NP_q(C_f^*(\chi_1,s))=\GNP_C(\Delta,p)$.

\noindent (2) Proof of this part is similar to the above by 
applying Proposition \ref{P:gNP} (2).
\end{proof}

Liu and Wan proved in \cite{LW09}  
that if $\NP_{\pi^{a(p-1)}}C_f^*(\chi,s)=\HP_C(\Delta)$
holds for one nontrivial $\chi$ implies that it holds for all 
non-trivial $\chi$. Application of Theorem \ref{T:GNP} 
yields a similar statement for $\GNP_C(\Delta,p)$ in Theorem \ref{T:GNP2} below,
which proves a 2-variable case of Liu-Wan's
Conjectures 7.10 and 7.11 in \cite{LW09}.

\begin{theorem}[Independency of character]
\label{T:GNP2}
Let $\chi$ be any character $\chi: \Z_p\rightarrow \C_p^*$ of conductor 
$p^{m_\chi}$ for some $m_\chi\in\Z_{\geq 1}$. Let notation be as in 
Proposition \ref{P:gNP}.

\begin{enumerate}
\item 
For any $f\in\cU(\bar\Q)$ and for $p$ large enough
$\NP_{\pi_\chi^{a(p-1)}}(C^*_f(\chi,s))$
is independent of $m_\chi$,
and 
\begin{eqnarray*}
\NP_{\pi_\chi^{a(p-1)}} C^*_f(\chi,s)
=\NP_{T^{a(p-1)}} C_f^*(T,s)
=\GNP_C(\Delta,p).
\end{eqnarray*}
\item 
For any $f\in\cU^o(\Q)$ and for $p$ large enough
$\NP_{\pi_\chi^{p-1}}(C^*_f(\chi,s))$
is independent of $m_\chi$,
and 
\begin{eqnarray*}
\NP_{\pi_\chi^{p-1}} C^*_f(\chi,s)
=\NP_{T^{p-1}} C_f^*(T,s)
=\GNP_C(\Delta,p).
\end{eqnarray*}
\end{enumerate}
\end{theorem}

\begin{proof}
Notice that for any nontrivial character $\chi$,
$$ 
\NP_{\pi_\chi^{a(p-1)}}C_f^*(\chi,s)\geq 
\NP_{T^{a(p-1)}}C_f^*(T,s)\geq 
\GNP_C(\Delta,p).
$$
But 
$\NP_{\pi_{\chi_1}^{a(p-1)}}C_f^*(\chi_1,s)=\GNP_C(\Delta,p)$ by Theorem \ref{T:GNP}(1), thus the equalities hold above and that is
$\NP_{T^{a(p-1)}}C_f^*(T,s)=\GNP_C(\Delta,p)$. This proves Part (1).
Proof of Part (2) follows from same argument 
via applying Theorem \ref{T:GNP}(2).
\end{proof}

Combining this theorem and Remark \ref{Remark:vertices} one observes the 
following statements:

\begin{corollary}
(1) For any $f\in\cU(\bar\Q)$ we have
$\lim_{p\rightarrow\infty}\NP_{\pi_\chi^{a(p-1)}} C^*_f(\chi,s)
=\lim_{p\rightarrow\infty} \GNP_C(\Delta,p) = \HP_C(\Delta).$\\
(2) For any $f\in\cU^0(\Q)$ we have
$\lim_{p\rightarrow\infty}\NP_{\pi_\chi^{p-1}} C^*_f(\chi,s)
=\lim_{p\rightarrow\infty} \GNP_C(\Delta,p) = \HP_C(\Delta).$
\end{corollary}

\subsection{Generic slopes for higher characters and arithmetic progression}

We will use the Independency of character theorem \ref{T:GNP2} to compute 
the $q$-adic Newton slopes for any character $\chi_m$ of conductor $p^m$.

\begin{definition}[Formula for generic Newton polygons: higher order character]
\label{D:GNP_chi}
For any $n,i\geq 1$ and $t>1$ let
$\beta_{n,i}(t):=\frac{1}{t^{m-1}}(i-1+\frac{n}{D}+\frac{\varepsilon_n}{t-1})$.
Let $N_{n,i}(t):=i W_\Delta(n)$ if $1\leq i\leq t^{m-1}$;
and $N_{n,i}(t):=(2t^{m-1}-i)W_\Delta(n)$ if $t^{m-1}<i\leq 2t^{m-1}-1$.
Let 
\begin{eqnarray*}
\GNP_L(\Delta,m,t) 
:=
\NP_T\prod_{i=1}^{2t^{m-1}-1}\prod_{n\in I_D}
\left((1-(T^{\beta_{n,i}(t)}s))(1-(T^{2-\beta_{n,i}(t)}s))\right)^{N_{n,i}(t)}.
\end{eqnarray*}
Let $\GNP_C(\Delta,m,t)$ be the product of $\GNP_L(\Delta,m,t)$ and the algebraic polygon of $\frac{1}{(1-s)^2}$.
\end{definition}

\begin{remark}\label{Remark:two}
\begin{enumerate}
\item 	
One observes
$\beta_{n,1}(t)=\beta_n(t)$
and $\beta_{n,i}(t)=\frac{i-1+\beta_n(t)}{t^{m-1}}$ for all $i\in\Z_{\geq 1}$.
Also one notices that  $\GNP_L(\Delta,m,p)$ has length-($2d_1d_2p^{2(m-1)}$) 
\begin{eqnarray*}
\GNP_L(\Delta,m,p)=\coprod_{i=1}^{2p^{m-1}-1}
\coprod_{n\in I_D}
\left\{\beta_{n,i}(p),2-\beta_{n,i}(p)\right\}^{N_{n,i}}.
\end{eqnarray*}

\item
A direct computation shows that 
$$
\GNP_C(\Delta,m,t)=\coprod_{i=1}^{\infty}\coprod_{n\in I_D}
\{\frac{i-1+\beta_n(t)}{t^{m-1}},\frac{i+1-\beta_n(t)}{t^{m-1}}\}^{iW_\Delta(n)}.
$$

\item 
If we write $\GNP_L(\Delta,t)=\coprod_{j=1}^{d}\alpha_j$, then 
$$
\GNP_L(\Delta,m,t):=\coprod_{i=1}^{t^{m-1}-1}\coprod_{j=1}^{d}
(\frac{i-1+\alpha_j}{t^{m-1}})^{N(i)} 
$$  
where $N(i)=i$ for $1\leq i\leq t^{m-1}$ and $N(i)=2t^{m-1}-i$
for $t^{m-1}<i\leq 2t^{m-1}-1$. 
In particular, $\GNP_L(\Delta,m,p)$ is 
$p^m$-adic weighted arithmetic progression of $\GNP_L(\Delta,p)$.
\end{enumerate}
\end{remark}

In the following theorem, we prove the existence of global generic polynomials 
whose reduction at all but finitely many prime $p$ is generic.

\begin{theorem}
\label{T:character}
Let $\chi_m:\Zp\rightarrow \C_p^*$ be any character of conductor $p^m$ with $m\geq 1$. Let notation be as in Proposition \ref{P:gNP}.
\begin{enumerate}
\item 
For any $f\in\cU(\bar\Q)$, for $p$ large enough 
we have 
$$\NP_q(L_f^*(\chi_m,s)^{-1})=\GNP_L(\Delta,m,p)$$
for all $m\geq 1$.

\item 
Let $f\in \A(\Q)$.
Then for all $p$ large enough 
$$
\NP_p(L_f^*(\chi_m,s)^{-1})=\GNP_L(\Delta,m,p)
$$ 
for any $m\geq 1$ 
if and only if $f\in\cU^o(\Q)$.
\end{enumerate}
\end{theorem}
\begin{proof}
By Theorem \ref{T:GNP2}(1) for $f\in\cU(\bar\Q)$, 
$\NP_{\pi_{\chi_m}^{a(p-1)}}C_f^*(\chi_m,s)=\GNP_C(\Delta,m,p)$.
As $(q)=(\pi_{\chi_m}^{a(p-1)})^{p^{m-1}}$,
$\NP_q(C_f^*(\chi_m,s))=\coprod_{i=1}^{\infty}\prod_{n\in I_D}(\frac{i-1+\beta_n(p)}{p^{m-1}})^{i W_\Delta(n)}$.
By (\ref{E:relation1}), $\NP_q(L_f^*(\chi_m,s)^{-1})$ 
has precisely $2Dp^{2(m-1)}$ slopes in the following set in $[0,2]$
\begin{eqnarray*}
&&\coprod_{i=1}^{2p^{m-1}-1}\coprod_{n\in I_D}
(\frac{i-1+\beta_n(p)}{p^{m-1}})^{iW_\Delta(n)} -
\coprod_{i=1}^{p^{m-1}+1}\coprod_{n\in I_D}
(1+\frac{i-1+\beta_n(p)}{p^{m-1}})^{2iW_\Delta(n)}.
\end{eqnarray*}
After cancellation, 
\begin{eqnarray*}
\NP_q(L_f^*(\chi_m,s)^{-1})=
\coprod_{i=1}^{2p^{m-1}-1}\coprod_{n\in I_D}
(\frac{i-1+\beta_n(p)}{p^{m-1}})^{N_{n,i}(p)}.
\end{eqnarray*}
By Remark \ref{Remark:two}, this proves Part (1).
Proof of Part (2) follows the same argument and Theorem \ref{T:GNP2} (2).
\end{proof}

Finally we shall prove our main theorems (Theorems \ref{T:1} and 
Theorem \ref{T:arithmeticProg}) of this paper below.

\begin{theorem}[Main Theorem]
\label{T:6}
Fix $d_1,d_2\geq 3$ that defines $\Delta$.
Let $\A_\Delta$ be the (coefficient) parameter space of all polynomial 
$f=\sum_{\bv\in \Delta} a_{\bv} \bx^\bv$ with $(x_1,x_2)$-degree equal to $(d_1,d_2)$. Let $\GNP_L(\Delta,m,t)$ and $\GNP_C(\Delta,m,t)$
be defined in Definition \ref{D:GNP_chi} for nontrivial residue class $\cR$.
\begin{enumerate}
\item
There is a Zariski dense open subset $\cU_\Delta$ in $\A_\Delta$ defined over $\Q$ such that,
for all $f\in \cU_\Delta(\bar\Q)$ and prime $p$ large enough in $\cR$
we have 
\begin{eqnarray*}
\NP_q L_f^*(\chi_m,s)^{-1}&=&\GNP_L(\Delta,m,p)\\
\NP_{\pi_{\chi_m}^{a(p-1)}}(C_f^*(\chi_m,s))
&=&\NP_{T^{a(p-1)}}(C_f^*(T,s))
=\GNP_C(\Delta,m,p).
\end{eqnarray*}
for any character $\chi_m:\Z_p\rightarrow \Cp^*$ of conductor $p^m$ 
with $m\geq 1$.
\item
There is a Zariski dense open subset $\cU_\Delta^o$ in 
$\A_\Delta$ defined over $\Q$ such that,
for all $f\in \cU_\Delta^o(\Q)$ and prime $p$ large enough in $\cR$
we have 
\begin{eqnarray*}
\NP_p L_f^*(\chi_m,s)^{-1}&=&\GNP_L(\Delta,m,p)\\
\NP_{\pi_{\chi_m}^{p-1}}(C_f^*(\chi_m,s))
&=&\NP_{T^{p-1}}(C_f^*(T,s))
=\GNP_C(\Delta,m,p).
\end{eqnarray*}
for any character $\chi_m:\Z_p\rightarrow \Cp^*$ of conductor $p^m$ 
with $m\geq 1$.
\end{enumerate}
\end{theorem}

\begin{proof}
Consider the natural projection map $\iota:\A_\Delta\rightarrow \A$
defined as follows:
For any $f$ in $\A_\Delta$, represented by $(a_\bv)_{\bv\in\Delta\cap\Z^2}$
such that $a_{d_1,d_2}, a_{d_1,0},a_{0,d_2}\neq 0$, 
is sent via $\iota$ to $(a_\bv)_{\bv\in\Delta\cap\Z^2}$
such that $a_{d_1,d_2}=a_{d_1,0}=a_{0,d_2}=1$.
Let $\cU_\Delta:=\iota^{-1}(\cU)$ and $\cU_\Delta^o:=\iota^{-1}(\cU^o)$.
We have $\cU_\Delta\subset \cU_\Delta^o$ are open dense subsets in $\A_\Delta$.
Then our statement follows from that in 
Theorem \ref{T:character}.
\end{proof}

\begin{remark}
\label{R:mainthereom}
One can prove this theorem for $\A_\Delta$ directly via similiar argument 
as that for that of $\A$ above in Theorem \ref{T:character}, which would 
be able to give an explicit description of the set $\cU_\Delta$ in an 
analogous fashion as that in Proposition \ref{P:gNP}.
We leave the details to the interested readers.
\end{remark}

\section{Generic eigencurves for $2$-variable Artin-Schreier-Witt towers }

Eigencurve for Artin-Schreier-Witt towers were introduced 
in \cite[Section 4]{DWX16}. 
Let $\cW_p$ denote the rigid analytic open unit disc associated to $\Zp[[T]]$.
Recall that eigencurve $\cA_f$ associated to the Artin-Schreier-Witt tower for 
$f=\sum_{\bv\in\Delta\cap\Z^2} a_\bv \bx^\bv$ 
is the zero locus of $C_f^*(T,s)$, viewed as a rigid analytic subspace of 
$\cW_p\times \Gmrig$, 
where $s$ is the coordinate of the second factor. Denote the natural projection 
to the first factor by $\wt_p: \cA_f\rightarrow \cW_p$.
Davis-Wan-Xiao showed that for a fixed prime $p$, when $f$ is one-variable of degree $d$, that for $m_\chi$ large enough 
the eigencurve $\cA_f$ is disjoint union of finite flat cover of the weight space $\cW_p$
each is zero locus of degree $d$.

\begin{proposition}\label{P:eigen}
Let notation be as in Theorem \ref{T:6}.
Let $f\in\cU_\Delta(\bar\Q)$ and let $p$ be large enough.  
Suppose for any $n\in I_D$ we assume that $\varepsilon_n\neq 0$.
\begin{enumerate}
\item Then 
$C_f^*(T,s)=\prod_{i=1}^{\infty} \prod_{n\in I_D} P_{n,i}(s)P_{n,i}^\vee(s)$
where $P_{n,i}(s)$ and $P_{n,i}^\vee(s)$ are polynomials 
in $\Zp[[T]][s]$ of degree  $i W_\Delta(n)$,
pure of $T^{a(p-1)}$-adic slope $i-1+\frac{n}{d_1d_2}
+\frac{\varepsilon_n}{p-1}$ and $i+1-\frac{n}{d_1d_2}-\frac{\varepsilon_n}{p-1}$,
 respectively. 

\item The eigencurve $\cA_f$ is an infinite disjoint union 
$\cA_f=\coprod_{i=1}^{\infty}\coprod_{n\in I_D}\cA_{f,n,i} \cA_{f,n,i}^\vee$ 
where each $\cA_{f,n,i}$ (and $\cA_{f,n,i}^\vee$ respectively) is of the zero locus of 
$P_{n,i}(s)$ (and $P_{n,i}^\vee(s)$ respectively).
\end{enumerate}
\end{proposition}

\begin{proof}
By Theorem \ref{T:6}, the
slopes for the eigencurves $\cA_f$ are 
$\NP_{T^{a(p-1)}}C^*_f(T,s)=\GNP_C(\Delta,p)$.
Since $\varepsilon_n=0$ for all $n\in I_D$, 
and $p$ is large enough, all slopes in $\GNP_C(\Delta,p)$
described in Remark \ref{Remark:vertices}
are distinct. 
By elementary $p$-adic Weierstrass preparation theorem, 
$C_f^*(T,s)$ split accordingly. 
Let $P_{n,i}$ be its factor polynomial 
corresponding to slope-($i-1+\beta_n(p)$) segment, and 
$P_{n,i}^\vee$ to slope-($i+1-\beta_n(p)$) segment, 
both of degree $i W_\Delta(n)$. 
This proves Part (1). 
Part (2) follows formally by setting $\cA_{f,n,i}$ for 
the zero locus of $P_{n,s}(s)$ and $\cA^\vee_{f,n,i}$
for that of $P^\vee_{n,i}(s)$.
\end{proof}

The above splitting of $\cA_f$ depends on 
the shape of the $T$-adic Newton polygon of 
$C_f^*(T,s)$. 
When $f$ is not generic,
the Newton polygon goes up in general
and subsequently its breaking points will change accordingly.
The only permanent vertex points in this situation
are the right endpoint of the length-1 slope-0 segment,
and the left endpoint of the length-1 slope-2 
segment in $L_f^*(T,s)$. Slopes in between are 
in the interval $(0,2)$. This observation gives rise to the following 
theorem:

\begin{theorem}[Theorem \ref{T:eigen}]
\label{T:eigen4}
Let notation be as in Proposition \ref{P:eigen}.
Let $f\in \cU_\Delta(\bar\Q)$.
Then $\cA_f$ is an infinite disjoint union: 
$$
\cA_f=\cA_{f,0}\coprod_{i=1}^{\infty} \cA_{f,i}
$$
where $\cA_{f,i}$ is the zero locus of a polynomial 
of degree $(2i-1)d_1d_2+\delta_\Delta+1$ with $T$-adic slopes lying
in the interval $(i-1,i]$ with $\delta_\Delta=0$ 
or $W_\Delta(D)$, and $\cA_{f,0}$ is a closed point. 
\end{theorem}

\begin{proof}
Write $D=d_1d_2$.
Let $\delta_\Delta=W_\Delta(D)$ if $\varepsilon_{D}=0$ and 
$\delta_\Delta=0$ otherwise, so $2\delta_\Delta$ is 
the length of slope-1 segment in 
$\GNP_L(\Delta,p)$. 
Now our statement follows
by counting the slopes multiplicities in $\GNP_C(\Delta,p)$
(see Remark \ref{Remark:vertices}) following 
the paragraph above this theorem.
\end{proof}

The proof of Theorem \ref{T:eigen} does not take advantage of any special property of 
$\GNP_C(\Delta,p)$, this indicates that an analog of this theorem 
may hold generally for all prime $p$ and all $2$-variable $f$ in $\A(\bar\Q)$
when $m$ large enough, namely, $\cA_f$  
is a disjoint union of covers of degree 
$N,N+2d_1d_2, N+4d_1d_2, N+6d_1d_2,\cdots$
for some $N$ (with one exceptional part) over the weight space.

\end{document}